\documentclass[11pt]{article}
\usepackage{enumitem}

\usepackage{authblk}
\usepackage[dvipsnames]{xcolor}
\usepackage{wrapfig}
\usepackage{tabularray}
\usepackage{caption}
\usepackage[mathscr]{eucal}
\usepackage{hyperref}
\usepackage[utf8]{inputenc}
\usepackage{graphicx}
\usepackage{mathtools}
\usepackage{amssymb}
\usepackage{amsmath}
\usepackage{amsthm}
\usepackage{bm}
\usepackage[normalem]{ulem}

\newtheorem{theorem}{Theorem}
\newtheorem{lemma}[theorem]{Lemma}

\newtheorem{corollary}[theorem]{Corollary}

\theoremstyle{definition}
\newtheorem{definition}[theorem]{Definition}
\theoremstyle{remark}
\newtheorem{remark}[theorem]{Remark}

\newcommand{\R}{\mathbb{R}}

\newcommand{\pr}{\partial}
\newcommand{\mext}{M_{\textrm{ext}}}

\usepackage{vmargin} 
\setmargins{2.5cm}{1.5cm}{16.5cm}{23.42cm}{10pt}{1cm}{0pt}{2cm}

\date{\today}

\title{$C^{0}$-inextendibility of FLRW spacetimes within a subclass of axisymmetric spacetimes}

\author{Melanie Graf\thanks{Department of Mathematics, Universit\"at Hamburg,  20146 Hamburg, Germany; \url{melanie.graf@uni-hamburg.de} \\
Significant parts of this work was completed while at the Department of Mathematics, Universit\"at T\"ubingen, 72076 T\"{u}bingen, Germany and at the Institute for Mathematics, Universit\"at Potsdam, 14476 Potsdam, Germany.} \;
and Marco van den Beld-Serrano\thanks{Department of Mathematics, Universit\"at Regensburg, 93053 Regensburg, Gemany; \url{marco.van-den-beld-serrano@ur.de}. Most part of this work was completed while at the Department of Mathematics, Universit\"at T\"ubingen, 72076 T\"{u}bingen, Germany.}}

\begin{document}
\date{}
\maketitle

\vspace{-.3in}

\begin{abstract} 

Starting from the proof of the $C^0$-inextendibility of Schwarzschild by Sbierski, the past decade has seen renewed interest in showing low-regularity inextendibility for known spacetime models. 
Specifically, a lot of attention has been paid to FLRW spacetimes and there is an ever growing array of results in the literature. Apart from hoping to provide a concise summary of the state of the art we present an extension of work by Galloway and Ling on $C^0$-inextendibility of certain FLRW spacetimes within a subclass of spherically
symmetric spacetimes,  \cite{GallowayLing}, to $C^0$-inextendibility within a subclass of axisymmetric spacetimes. Notably our result works in the case of flat FLRW spacetimes with $a(t)\to 0$ for $t\to 0^+$, a setting where other known $C^0$-inextendibility results for FLRW spacetimes due to Sbierski, \cite{Sbierski2023}, do not apply.

\medskip

\end{abstract}

\setcounter{tocdepth}{2}
\tableofcontents

\section{Introduction}

(In-)extendibility of spacetimes is a recurrent topic well known to play a central role in General Relativity both for specific models and for quite general classes. The usual procedure is to follow one (or both) of the following paths: either an explicit extension of the spacetime is found/constructed or it is proven that there is some general obstruction to extendibility within a certain class of extensions. For $C^2$-inextendibility blow up of some  curvature scalar (e.g., the scalar curvature or the Kretschmann scalar) gives such an extendibility obstruction because these scalars are invariant under diffeomorphisms. In other words, their blow up is a geometric property (i.e. coordinate independent) of the spacetime. Of course, the blow up of a curvature scalar only tells us that the manifold is $C^{2}$-inextendible (as the curvature scalars contain second order derivatives of the metric). Hence, different strategies are required in order to explore the inextendibility of a spacetime in a lower regularity class (e.g. $C^{0}$- or $C^{0,1}$-regularity). Here a lot of new tools and techniques have been developed in the last six years. In particular, the question of $C^{0}$-inextendibility was first
tackled by Sbierski \cite{Sbierski2016}, who proved that the Minkowski and the maximally extended Schwarzschild spacetime are $C^{0}$-inextendible. We now have a collection of low regularity inextendibility criteria the most stringent (though not always most useful in practice as one is often interested in inextendibility precisely in the cases where spacetimes do contain incomplete geodesics) of them being timelike geodesic completeness: in the first place, in \cite{Sbierski2016} it was proven that if no timelike curve intersects the boundary of the extension $\pr \iota(M)$, then the spacetime is inextendible. This result already pointed to the idea that, under certain additional assumptions, timelike (geodesic) completeness would yield the inextendibility of a spacetime (in a low regularity class). Indeed, in \cite{LingGalSbierski} it was proven that a smooth globally hyperbolic and timelike geodesically complete spacetime is $C^0$-inextendible. Moreover, in \cite{Graf}, it was shown that if the global hyperbolicity condition is dropped the spacetime still is $C^{0,1}$-inextendible. Finally, in \cite{Minguzzi2019}, using the Lorentz-Finsler framework, it was shown that a smooth timelike geodesically complete spacetime is $C^{0}$-inextendible and in \cite{GrantKunzingerSaemann} an inextendibility result for timelike complete Lorentzian length spaces is shown. 

Beyond these general structural results special attention has been paid to FLRW spacetimes. FLRW spacetimes generally and especially {\em flat} or {\em Euclidean} FLRW spacetimes are of great interest in Cosmology in order to model the universe. For this reason one is particularly interested in (in-)extendibility results across a possible Big Bang. Part of the purpose of this note is to given an overview of the different concrete cases treated and the precise results available depending on the behaviour of the scale factor as one approaches the Big Bang or future infinity\footnote{It is believed that the Universe underwent different evolution stages in which the scale factor $a$ would have had a different dependence on time (see the discussion in \cite[Section 5.2, Table 5.1]{Wald}).}  and whether we are looking at spherical, Euclidean or hyperbolic FLRW spacetimes. This general review will be the focus of Section \ref{sec:inext_overview}, see also Table \ref{table} for a quick summary.  

The most open of the three cases is that of Euclidean FLRW spacetimes. While there is a $C^{0,1}_{\textrm{loc}}$-inextendibility result based on local holonomy, see \cite{Sbierski2020}, that works for flat FLRW spacetimes with a Big Bang at finite time and particle horizon the question of $C^0$-(in-)extendibility is still very open. Galloway and Ling in \cite{GallowayLing} used an approach based on uniqueness  of certain coordinate changes in flat and hyperbolic FLRW spacetimes to prove that such spacetimes are $C^{0}$-inextendible within a certain subclass of spherically symmetric spacetimes. We present this result in Section \ref{sec:sphericalresult}. Apart from reviewing available literature  the aim of this paper is to generalize their result to $C^{0}$-inextendibility within a larger subclass of axisymmetric spacetimes which enjoy similar uniqueness properties for coordinates as the strongly spherically symmetric extensions considered in \cite{GallowayLing}. 
 We show (cf. Corollary \ref{cor:noaxisymmetric_Flat_FLRW})
 \begin{theorem}
    Let $(M,g)$ be a (4-dimensional) flat future inextendible FLRW spacetime with scale factor $a$ satisfying that $a'(0)\in (0,\infty]$. Then, there is no past natural strongly axisymmetric $C^{0}$-extension of $(M,g)$ compatible with the strongly spherically symmetric coordinates of \cite{GallowayLing}. 
\end{theorem}
referring to Section \ref{sec:cylindrical} for the definitions used here and to Corollary \ref{cor:noaxisymmetric_Flat_FLRW} for the precise statement (see also Theorem \ref{theo:noaxissymmetric}). We further mention that just as in \cite{GallowayLing} a similar result is available for hyperbolic FLRW spacetimes (cf. Corollary \ref{cor:noaxisymmetric_Hyp_FLRW}).  

\subsection{Notations and conventions}

\par We conclude the introduction with fixing our notations and conventions. The term $C^{k}$ spacetime for us always denotes a connected time-oriented Lorentzian manifold $(M,g)$, where $M$ is a smooth manifold but the metric $g$ is merely $C^{k}$-regular. Furthermore, timelike curves are piecewise smooth curves whose right and left handed sided derivatives lie in the same connected component of the lightcone, which coincides with the convention in \cite{GallowayLing}. The following basic concepts play an important role in our study.

\begin{definition}[$C^{l}$-extension]\label{def:classicalextension}
    Fix $k\geq 0$ and let $0\leq l\leq k$. Let $(M,g)$ be a $C^{k}$ spacetime with dimension $d$. A $C^{l}$\emph{-extension} of $(M,g)$ is a proper isometric embedding $\iota$ 
\begin{equation*}
        \iota\;:\; (M,g)\hookrightarrow (M_{\textrm{ext}},g_{\textrm{ext}})
    \end{equation*}
where $(M_{\textrm{ext}},g_{\textrm{ext}})$ is $C^{l}$ spacetime of dimension $d$. If such an embedding exists, then $(M, g)$ is said to be \emph{$C^{l}$- extendible}. The topological boundary of $M$ within $M_{\textrm{ext}}$ is $\partial \iota(M)\subset M_{\textrm{ext}}$. By a slight abuse of notation we will sometimes also call $(M_{\textrm{ext}},g_{\textrm{ext}}$) the extension of $(M,g)$, dropping the embedding $\iota$.
\end{definition}

\begin{definition}[Future and past boundary]
    Let $\iota: (M,g)\hookrightarrow (M_{\textrm{ext}},g_{\textrm{ext}})$ be a $C^{l}$-extension. We define the \emph{future boundary} $\partial^{+} \iota (M)$ and \emph{past boundary} $\partial^{-} \iota (M)$:
\begin{equation*}
        \partial^{+}\iota(M)\coloneqq \{p\in \partial \iota(M)\;:\: \exists\; \textrm{f.d.t.l. curve }\; \gamma : [0, 1] \rightarrow M_{\textrm{ext}} \;\textrm{with}\; \gamma(1)=p,\; \gamma([0,1))\subset \iota(M) \}
\end{equation*}
\begin{equation*}
        \partial^{-}\iota(M)\coloneqq \{p\in \partial \iota(M)\;:\: \exists\; \textrm{f.d.t.l. curve}\; \gamma : [0, 1] \rightarrow M_{\textrm{ext}} \;\textrm{with}\; \gamma(0)=p,\; \gamma((0,1])\subset \iota(M) \}
\end{equation*}
where ``f.d.t.l. curve" stands for future directed timelike curve. If $\partial^{-}\iota(M)\neq \emptyset$ (or $\partial^{+}\iota(M)\neq \emptyset$), then $\iota$ is called a \emph{past $C^{l}$-extension} (resp. a \emph{future $C^{l}$-extension})) and if no such extension exists, $(M,g)$ is said to be \emph{past $C^{l}$-inextendible} (resp. \emph{future $C^{l}$- inextendible}).
\end{definition}

Two of the earliest results about $C^0$-extensions tell us that, while not every point point in $\partial\iota(M)$ has to belong to one of these boundary components (and neither do they have to be disjoint), it suffices to study the future and past boundary of a spacetime in order to prove its $C^{0}$-inextendibility:
\begin{lemma}[Lemma 2.17 in \cite{Sbierski2016}, Proposition 2.3 in \cite{GallowayLing}]
    Let $\iota : (M,g)\rightarrow (M_{\textrm{ext}},g_{\textrm{ext}})$ be a $C^{0}$-extension. Then $\partial^{+}\iota(M)\cup\partial^{-}\iota(M)\neq\emptyset$.
\end{lemma}
Further emptiness of one of these implies a rather nice structure for the other:
\begin{theorem}[Theorem 2.6 in \cite{GallowayLing}]\label{theo:pastboundary}
    Let $\iota$: $(M,g)\rightarrow (M_{\textrm{ext}},g_{\textrm{ext}})$ be a $C^{0}$-extension. If $\partial^{+}\iota(M)=\emptyset$, then $\partial^{-}\iota(M)$ is an achronal topological hypersurface.
\end{theorem}

As pointed out in the introduction we will be focusing on  FLRW spacetimes, that is isotropic\footnote{We remark that \cite{Avalos} shows that spacetime isotropy is necessary for the desired rigidity properties. Furthermore, for the definition of \emph{cosmological spacetime} see, e.g. the introduction of the aforementioned paper.} cosmological spacetimes which are necessarily of the following form: 
\begin{definition}[FLRW spacetimes]\label{def:FLRW}
  A ($d+1$)-dimensional \emph{FLRW spacetimes} is a warped product of an open interval $I\subset \mathbb{R}$ with a complete and simply connected $d$-dimensional Riemannian manifold with constant scalar curvature $K$. Depending on the value of $K$, the FLRW spacetimes are divided in three subgroups:
    \begin{align}
     &\textrm{\emph{Spherical }($K=+1$):} & &M=I\times \mathbb{S}^{d}, & & g=-dt^{2}+a^{2}(t)\left(dr^{2}+\sin^{2}(r)d\Omega^{2}_{d-1}\right)\\
    &\textrm{\emph{Hyperbolic }($K=-1$):} & &M=I\times\mathbb{R}^{d}, & &g=-dt^{2}+a^{2}(t)\left(dr^{2}+\sinh^{2}(r)d\Omega^{2}_{d-1}\right) \\
    &\textrm{\emph{Flat/Euclidean} ($K=0$):} & &M=I\times\mathbb{R}^{d}, & &g= -dt^{2}+a^{2}(t)\left(dr^{2}+r^{2}d\Omega^{2}_{d-1}\right)
    \end{align}
   where the metrics $g$ are written in \emph{FLRW coordinates} $(t,r,\omega) \in I\times(0,\infty)\times \mathbb{S}^{d-1}$ and the \emph{scale factor} $a: I \rightarrow(0,\infty)$ is a smooth function. 

     If in addition it holds that $\lim_{t\to t_{\textrm{inf}}^{+}}a(t)=0$, where $t_{\textrm{inf}}\coloneqq \textrm{inf}(I)$, then we call $(M,g)$ a FLRW spacetime \emph{with a Big Bang as $t\to t_{\textrm{inf}}^{+}$} and if $t_{\textrm{inf}}>-\infty$ we say the spacetime has a Big Bang at finite time. FLRW spacetimes with $\lim_{t\to t_{\textrm{inf}}^{+}}a(t)\neq0$ will be accordingly referred to as \emph{FLRW spacetimes without a Big Bang}. If $t_{\textrm{inf}}=-\infty$ we say the spacetime is past eternal (with or without a Big Bang) and we note that the case of a Big Bang at finite time is qualitatively different from a past eternal FLRW spacetime having a Big Bang.
\end{definition}

Unless stated otherwise, we will always assume that $I=(0,\infty)$ and that there is a Big Bang as $t\to 0^+$. More general intervals $I\subset \mathbb{R}$ will only be considered briefly in Section \ref{sec:inext_overview}, when discussing some past eternal FLRW spacetimes.

We can further classify FLRW spacetimes with a Big Bang into FLRW spacetimes with particle horizon and FLRW spacetimes without particle horizon depending on the integrability of $\frac{1}{a(t)}$ as $t\to 0^{+}$.
\begin{definition}[Particle horizon]\label{def:particle_hor}
    Let $(M,g)$ be an FLRW spacetime. It is said to have a \emph{particle horizon} provided $\int_{0}^{1} \frac{1}{a(t')}dt'<\infty$. Otherwise, it has \emph{no particle horizon}.
\end{definition}

In the next section we will present some criteria which guarantee the ($C^{0}$- or $C^{0,1}_{\textrm{loc}}$-)inextendibility of FLRW spacetimes. These criteria mostly depend on the considered type of FLRW spacetime (flat, spherical or hyperbolic) and on properties related to the asymptotic behaviour of the scale factor. 
Furthermore, some examples of relevant $C^{0}$-extendible FLRW spacetimes will be discussed. As having a particle horizon or not is related to the rate at which $a(t)$ approaches zero it is perhaps not surprising that this will play a role in the inextendibility results we are considering. However, as the discussion and 
in particular Table \ref{table} will show, this relationship does not appear to be straightforward in the sense that having or not having a particle horizon always leads to stronger inextendibility results.

\paragraph*{Acknowledgements}  This article originally started from work on MvdBS' Masters thesis written at the University of Tübingen. We would like to thank Carla Cederbaum for her support and bringing this collaboration together. We would further like to thank Eric Ling for bringing some of these problems to our attention and stimulating discussions. 
MG acknowledges the support of the German Research Foundation through the SPP2026 ''Geometry at Infinity'' and the Cluster of Excellence EXC 2121 “Quantum Universe”, the University of Tübingen and the University of Potsdam. MvdBS thanks Carla Cederbaum for her financial support during the development of this research project, the Studienstiftung des deutschen Volkes for granting him a scholarship during his Master studies and for his PhD project, and Felix Finster for his support.

\section{Overview of the low-regularity inextendibility results for FLRW spacetimes}\label{sec:inext_overview}
The aim of this section is to briefly review the main inextendibility results that have been proven for the different types of $(d+1)$-dimensional FLRW spacetimes. In Table \ref{table} the discussed low regularity inextendibility results are collected. We note that we will largely  discuss future inextendibility and past inextendibility separately, focusing on past inextendibility. Of course time dual versions of all results hold as well, however future inextendibility results will always be for $I=(c,\infty)$, with $c\in\mathbb{R}$, and mostly depend on the behaviour of $a(t)$ as $t\to \infty$ so that their time duals will not be applicable to Big Bangs at finite time (and vice versa).

\subsection{Future inextendibility}

\par Before moving on to past inextendibility we start by stating a future $C^{0}$-inextendibility result, the proof of which is essentially (though in the case of (i) somewhat anachronistically as the original proof preceded the general result about timelike geodesic completeness implying $C^0$-inextendibility from \cite{Minguzzi2019}) based on showing that the given conditions on the scale factor already guarantee future timelike completeness. The importance for us  will be that future inextendibility guarantees that the past boundary satisfies some nice dimensional and causal properties via Theorem \ref{theo:pastboundary}.

\begin{theorem}[Theorem 3.2 in \cite{GallowayLing} and Theorem 3.5 in \cite{Sbierski2020}]\label{theo:nofuture}
    \label{FLRWfutureboundary}
    Let $(M,g)$ be a ($d+1$)-dimensional FLRW spacetime (with $d\geq1$) which satisfies one of the following conditions:
    \begin{enumerate}[label=(\roman*)]
        \item $(M,g)$ is a flat or hyperbolic FLRW spacetime and the scale factor has a sublinear growth, i.e. $a(t)\leq mt+b$ for all $t\in (0,\infty)$ with $m>0$ and $b\geq0$, and $a'(t)>0$ for all $t\in(0,\infty)$.
        \item $\int_{1}^{\infty}\frac{a(t)}{\sqrt{a(t)^{2}+1}}dt=\infty$.
    \end{enumerate} 
    Then, $(M,g)$ is future $C^{0}$-inextendible and given any past $C^{0}$-extension $\iota : M\rightarrow \mext$, the past boundary $\partial^{-}\iota(M)$ is an achronal topological hypersurface. 
\end{theorem}
Note that (as pointed out in \cite{Sbierski2020}) the condition that $\int_{1}^{\infty}\frac{a(t)}{\sqrt{a(t)^{2}+1}}dt=\infty$ is in particular satisfied if $\lim_{t\to\infty}a(t)\neq0$.

\subsection{Past inextendibility}
While there exist quite general future $C^{0}$-inextendibility results for the FLRW spacetimes, most past $C^{0}$-inextendibility results depend on the particular type of FLRW spacetime we are considering (i.e. spherical, hyperbolic or flat), on the existence of particle horizons or on the symmetry properties of the considered extensions. Before presenting some of the strongest past $C^{0}$-inextendibility results for each type of FLRW spacetimes, we state the following past $C^{0,1}_{\textrm{loc}}$-inextendibility result by Sbierski that applies to any FLRW spacetime with particle horizon.

\begin{theorem}[Theorem 3.7 in \cite{Sbierski2020}]
     Let $(M,g)$ be a ($d+1$)-dimensional FLRW spacetime ($d\geq1$) with particle horizon. Then $(M,g)$ is past $C^{0,1}_{loc}$-inextendible.
\end{theorem}
The proof is based on estimates involving blow-up of some local holonomy. More precisely Sbierksi shows that there exists a specific sequence of loops which approach the past boundary such that the parallel transport map  along them is unbounded. This contradicts Lipschitz extendibility because, provided the metric components and its first order derivatives are bounded, the parallel transport map along curves in a bounded domain with a uniformly bounded tangent vector is uniformly bounded (which Sbierski proves using a Grönwall's inequality type of argument, cf. \cite[Lemma 2.19]{Sbierski2020}).

With respect to the past $C^{0}$-inextendibility criteria for general FLRW spacetimes, note that we will focus on the case that $d\geq 2$. This is due to the fact that for $d=1$ any FLRW spacetime is past $C^{0}$-extendible (initially proven in Section 3.2 in \cite{GallowayLing}, see also Section 6 in \cite{Sbierski2022} or Section 1.2 in \cite{Sbierski2023}). 

\par The main $C^{0}$-inextendibility result available for spherical and hyperbolic FLRW spacetimes (without particle horizon) is the following.
\begin{theorem}[Theorem 1.5 and 1.6 in \cite{Sbierski2023}]\label{th:Sbierskis_inext}
    Let $(M,g)$ be a ($d+1$)-dimensional spherical FLRW spacetime (with $d\geq2$) without particle horizon. Assume one of the following conditions holds:
    \begin{enumerate}[label=(\roman*)]
        \item $(M,g)$ is a spherical ($K=+1$) FLRW spacetime.
        \item $(M,g)$ is a hyperbolic ($K=-1$) FLRW spacetime with $a(t)e^{\int_{t}^{1}\frac{1}{a(t')}dt'}\to \infty$ as $t\to0^{+}$.
    \end{enumerate}
    Then, $(M,g)$ is past $C^{0}$-inextendible.
\end{theorem}
The proof of (i) uses ideas from \cite{Sbierski2016} and causality theoretic arguments while (ii) makes use of characterizing TIFs\footnote{''TIF'' stands for terminal indecomposable future sets, that is future sets $P$ (i.e. sets satisfying $I^+(P)\subseteq P$) which cannot be decomposed into two open past proper subsets and does not coincide with the past of any point in $M$.} in these spacetimes and  finding a parametrization for the past boundary and relating TIFs to points on the past boundary.
We remark that in condition (ii) of the previous theorem, that $a(t)e^{\int_{t}^{1}\frac{1}{a(t')}dt'}\to \infty$ as $t\to0^{+}$ plays a very important role: if $a(t)e^{\int_{t}^{1}\frac{1}{a(t')}dt'}\to (0,\infty)$ as $t\to0^{+}$, $\lim_{t\to0^{+}}a'(t)=1$ and $a(t)$ has a sublinear growth (i.e. $(M,g)$ is a so called \emph{Milne-like spacetime}), then the hyperbolic FLRW spacetimes do admit a $C^{0}$-extension (cf. the next subsection).

The case of flat FLRW spacetimes is more open: While there is a recent paper by Ling, \cite{Ling2024}, based on Sbierski's work in the spherical and hyperbolic setting that does consider the flat case, he looks at past eternal flat FLRW spacetimes with a Big Bang as $t\to -\infty$ as opposed to a Big Bang at finite time.

\begin{theorem}[Theorem 1.1 in \cite{Ling2024}]
Let (M, g) be a  flat simply connected FLRW spacetime $M=\mathbb{R}\times\mathbb{R}^{3}$ with $d\geq 2$. Moreover, suppose the scale factor satisfies that
\begin{enumerate}[label=(\roman*)]
    \item $\lim_{t\to-\infty}a(t)=0$.
    \item $\lim_{t\to-\infty}a(t)\int_{t}^{1}\frac{1}{a(t')}dt'=\infty$
\end{enumerate}
Then $(M,g)$ is past $C^{0}$-inextendible.
\end{theorem}
As already pointed out by Ling these two conditions on the scale factor imply that these spacetimes have no particle horizon. Similarly to the hyperbolic setting condition (ii) appears to be sharp: If the scale factor satisfies (i) and the spacetime has no particle horizon  but $\lim_{t\to-\infty}a(t)\int_{t}^{1}\frac{1}{a(t')}dt'\in (0,\infty)$ then  it does admit a $C^0$-extension (cf. \cite{GeshnizjaniLingQuintin} and the next subsection).

Returning to our discussion of results on FLRW spacetimes with a Big Bang at finite time  there unfortunately does at this point not exist a general $C^{0}$-inextendibility result ('general' in the sense of only imposing conditions on the scale factor) for flat FLRW spacetimes. In the next section we will present a $C^{0}$-inextendibility result by Ling and Galloway \cite{GallowayLing} that restricts the symmetry class of the metric of the extended spacetime. In particular, they prove the following result:
\begin{corollary}[Corollary 4.2 in \cite{GallowayLing}]
    Let $(M,g)$ be a flat FLRW spacetime with $a'(0) \in (0,\infty]$ and which is future inextendible (for instance because it is future timelike geodesically complete or because it satisfies condition (i) in Theorem \ref{theo:nofuture}). Then there exists no natural strongly spherically symmetric $C^{0}$-extension of $(M,g)$.
\end{corollary}
Note that we have adapted the statement slightly from its original in \cite{GallowayLing}, where they demand condition (i) in Theorem \ref{theo:nofuture}, to directly focus on the relevant consequence of this condition in the proof, namely a way to establish emptiness of the future boundary, which we now know to follow e.g. from future timelike geodesic completeness.
We will introduce the relevant concepts and discuss this result more in depth in the next section. Note that in \cite{GallowayLing} an analogous result was also proven for hyperbolic FLRW spacetimes with the additional assumption that $\lim_{t\to0^{+}}a'(t)\neq 1$:
\begin{corollary}[Corollary 4.4 in \cite{GallowayLing}]\label{cor:nospherical_hyp_FLRW}
    Let $(M,g)$ be a future inextendible hyperbolic FLRW spacetime with $a'(0)\in [0,\infty]$ and $a'(0)\neq 1$. Then there exists no natural strongly spherically symmetric 
    $C^{0}$-extension of $(M,g)$.
\end{corollary}

\par Building on the results by Galloway and Ling, in Section \ref{sec:cylindrical} we will prove that if flat and hyperbolic FLRW spacetimes satisfying some conditions on the scale factor $a$ do admit a $C^{0}$-extension, then the extended spacetime cannot belong to a certain subclass of axisymmetric spacetimes (for the proof see Corollaries \ref{cor:noaxisymmetric_Flat_FLRW} and \ref{cor:noaxisymmetric_Hyp_FLRW}):
\begin{theorem}\label{theo:noaxissymmetric}
    Let $(M,g)$ be a future inextendible (4-dimensional) FLRW spacetime. Assume one of the following conditions holds:
    \begin{enumerate}[label=(\roman*)]
        \item $(M,g)$ is a flat FLRW spacetime and the scale factor $a$ satisfies that $a'(0)\in (0,\infty]$.
        \item $(M,g)$ is a hyperbolic FLRW spacetime and the scale factor $a$ satisfies that $a'(0)\in [0,\infty]$ and $a'(0)\neq1$.
    \end{enumerate}
    Then, there is no natural strongly axisymmetric $C^{0}$-extension of $(M,g)$ compatible with the strongly spherically symmetric coordinates appearing in \cite{GallowayLing}. 
\end{theorem}

\begin{table}
\centering
\begin{tblr}{
  row{1} = {c},
  cell{2}{1} = {r=2}{c},
  cell{2}{3} = {r=2}{},
  cell{4}{1} = {r=3}{c},
  cell{4}{3} = {r=5}{},
  cell{7}{1} = {r=2}{c},
  vlines,
  hline{1-2,4,9} = {-}{},
  hline{3,5-6,8} = {2}{},
  hline{7} = {1-2}{},
}
K   & Past inextendibility (Big Bang at finite time $t\to0^{+}$)    & Future inextendibility ($d\geq 1$)     \\
+1  & {\textbf{\textbf{No particle horizon:}}\\$C^0$-inext. ($d\geq 2$)}                                                                                          & $C^0$-inext. if $\int^\infty_1\frac{a(t)}{\sqrt{a(t)^{2}+1}}dt=\infty$~~                                  \\
    & {\textbf{\textbf{With particle horizon:}}\\ 
    $C^{0,1}_\textrm{loc}$-inext. ($d\geq 1$)}                                                                         &                                                                                                           \\
~-1 & {\textbf{\textbf{No particle horizon:}}\\$C^0$-inext. if $a(t)e^{\int_{t}^{1}\frac{1}{a(t')}dt'}\to \infty$ as $t\to0^{+}$ ($d\geq 2$)}                               & {$C^0$-inext.\ if $\int^\infty_1\frac{a(t)}{\sqrt{a(t)^{2}+1}}dt=\infty$\\ \\ \\ $C^0$-inext.\ if $a(t)\leq mt+b$, $a'(t)>0$} \\
    & {\textbf{With particle horizon:}\\$C^{0,1}_\textrm{loc}$-inext. ($d\geq 1$)}                                                                                  &                                                                                                           \\
    & {\textbf{\textbf{\textbf{\textbf{With symmetry assumptions:}}}}\\No nat. str. spher. symm. $C^{0}$-extensions ($d\geq2$)\\No nat. str. axisymm. $C^0$-extensions ($d=3$).\textbf{\textbf{\textbf{\textbf{}}}}}        &                                                                                                           \\
0   & {\textbf{\textbf{With symmetry assumptions:}}\\ No nat. str. spher. symm. $C^0$-extensions ($d\geq2$)\\No nat. str. axisymm. $C^0$-extensions ($d=3$)} &                                                                                                           \\
    & {\textbf{\textbf{With particle horizon:}}\\$C^{0,1}_\textrm{loc}$-inext. ($d\geq 1$)}                                                                         &                                                                                                           
\end{tblr}
\caption{Main low regularity inextendibility results for the different types of $(d+1)$-dimensional FLRW spacetimes with a Big Bang at finite time $t\to0^{+}$. Recall that for $d=1$ all FLRW spacetimes are past $C^{0}$-extendible.}
\label{table}
\end{table}

\subsection{Some important examples of \texorpdfstring{$C^{0}$}{C0}-extendible FLRW spacetimes}\label{sec:examples}
Before closing this section, we present some of the main examples for $C^0$-extendible hyperbolic and flat FLRW spacetimes: The common ground of these examples is that their scale factor behaves to some degree similarly to the behaviour of the scale factor when expressing certain proper open subsets of Minkowski or de Sitter spacetime in FLRW form. We want to reiterate that Minkowski and de Sitter (as well as anti de Sitter) spacetime themselves are all $C^0$-inextendible: While the first proofs of these facts employed slightly different methods (and can be found in \cite[Theorem 3.1]{Sbierski2016} for Minkowski, \cite[Remark/Outlook 1.3 i)]{Sbierski2016} for de Sitter and \cite[Theorem 2.12]{GallowayLing} for anti de Sitter), this now anachronistically follows immediately from them being timelike geodesically complete thanks to  \cite{Minguzzi2019}.

Our main examples for $C^0$-extendible hyperbolic and flat FLRW spacetimes are: 
\begin{itemize}
       \item \underline{\emph{Milne (and Milne-like) spacetimes}:}
    \par The Milne-like spacetimes (cf. \cite[Definition 3.3]{GallowayLing}) are hyperbolic FLRW spacetimes $M=(0,\infty)\times\mathbb{R}^{d}$ without particle horizon and with a scale factor that additionally satisfies that:
        \begin{enumerate}[label=(\roman*)]
            \item $a(t)$ has a sublinear growth and $a'(t)>0$ for all $t\in (0,\infty)$.
            \item $\lim_{t\to0^{+}}a'(t)=1$.
            \item $\lim_{t\to0^{+}}a(t)e^{\int_{t}^{1}\frac{1}{a(s)}ds}\in(0,\infty)$.
        \end{enumerate}
    Milne-like spacetimes always admit $C^{0}$-extensions (\cite[Theorem 3.4]{GallowayLing}) but depending on the exact form of the scale factor near $t=0$ not necessarily $C^{2}$ or $C^{\infty}$ extensions. Note that the 'classical' Milne spacetime (a hyperbolic FLRW spacetime $M=(0,\infty)\times\mathbb{R}^{d}$ with 
    $a(t)=t$ which is well-known to be isometric to $I^+(0)\subseteq (\R^{d+1},\eta)$ and hence even smoothly extendible) is an example of a Milne-like spacetime (for the physical motivation and properties of the Milne-like spacetimes, see \cite{Ling2020}).
     \item \underline{\emph{Hyperbolically sliced patches of de Sitter spacetime}:}
    Various patches of de Sitter spacetime can be realized as a spherical, hyperbolic or flat FLRW spacetime depending on the chosen coordinates (and their corresponding domain of definition). See \cite{Pascu} (Table 15 for the FLRW coordinates) or \cite[Section 3]{KimOhPark} for a nice account of the plethora of different coordinates for the de Sitter spacetime.
    \par In particular the hyperbolic FLRW spacetime $M=(0,\infty)\times \mathbb{H}^{3}$ with scale factor $a(t)=\sinh^{2}(t)$ -- and thus a Big Bang as $t\to0^{+}$ -- isometrically corresponds to a proper  open patch of de Sitter spacetime (covering a quarter of full de Sitter) and as such is smoothly extendible. An explicit $C^{0}$-extension where the extended manifold is conformal to the Minkowski spacetime was constructed in \cite[Section 3.2]{Ling2020}.
    \item \underline{\emph{Flatly sliced patches of de Sitter (and flat past asymptotically de Sitter) spacetimes}:}
    \par  The $4$-dimensional flat FLRW spacetime $M=\mathbb{R}\times \mathbb{R}^{3}$ with scale factor $a(t):=e^{t}$ is also isometric to a proper open patch of de Sitter spacetime (covering half of full de Sitter). The scale factor exhibits a Big Bang as $t\to -\infty $ and increases exponentially with time. Clearly this is smoothly extendible to full de Sitter as well. It is also $C^{0}$-extendible (see \cite[Section 3.2]{GeshnizjaniLingQuintin}) and similar to the Milne case the procedure for finding a $C^0$-extension generalizes to a family of \emph{flat past asymptotically de-Sitter spacetimes} (following the terminology of \cite[Definition 2.5]{GeshnizjaniLingQuintin} these are $4$-dimensional  flat FLRW spacetimes with a Big Bang as $t\to-\infty$ where the Hubble parameter $\frac{a'}{a}$ behaves similar enough to that of the flat slicing of de Sitter as $t\to -\infty$), cf. \cite[Corollary 3.6]{GeshnizjaniLingQuintin}. 
    They further give an explicit example of a scale factor satisfying a sufficient condition for $C^0$-extendibility but where the spacetime is no longer $C^2$ extendible (cf. \cite[Example (3.24)]{GeshnizjaniLingQuintin}.
\end{itemize}

\section{Summary of the proof of strongly spherically symmetric inextendibility in \texorpdfstring{\cite{GallowayLing}}{Galloway and Ling's paper}} \label{sec:sphericalresult}
In the rest of this paper only FLRW spacetimes with a Big Bang as $t\to0^{+}$ will be considered (and we will omit explicit reference to this initial singularity).

\begin{definition}[Strongly spherically symmetric spacetimes, cf.\ \cite{GallowayLing}]
\label{spherical}
    Let $(M,g)$ be a spacetime of dimension $d+1$. It is called a \emph{strongly spherically symmetric spacetime} if for all $p\in M$ there exists a neighbourhood $V$ of $p$ in $M$ and (potentially) a curve $L$ in $V$ such that on $U:=V\setminus L$
   there exists a diffeomorphism $\psi_s=(T,R,\omega):U\to \psi_{s}(U)\subset \mathbb{R}\times (0,\infty)\times  \mathbb{S}^2$ with $T: U\to \R$ , $R : U\rightarrow(0,\infty)$ and $\omega: U\rightarrow \mathbb{S}^{d-1}$ such that the pushforward of the metric takes the form:
    \begin{equation}
    \label{eq:Spher1}
         (\psi_{s})_{\ast} g=-F(T,R)dT^{2}+G(T,R)dR^{2}+R^{2}d\Omega^{2}_{d-1}
    \end{equation}
We call $(T,R,\omega)$ \emph{strongly spherically symmetric coordinates}. 
\end{definition}

\begin{remark} \label{naturalspherical_new} Note that we are allowing $R$ to be the timelike coordinate.
 For why we call this {\em strongly} spherically symmetric and a discussion relating it back to coordinate invariant descriptions of spherical symmetry we refer to the brief discussion in the beginning of Section 4 in \cite{GallowayLing}. The main point is that we are excluding potential $dTdR$ cross terms despite being in a regularity where the usual trick to do cannot be applied anymore. Our main reason for allowing the exclusion of a curve $L$ is that we will want to be explicitly restrict the range of our $R$ coordinate to $(0,\infty)$ for technical reasons. It 
 also allows us to start with the FLRW metrics in the coordinates from Definition \ref{def:FLRW} where we technically already had to exclude an origin from our spherical coordinates on the spatial slice.
\end{remark}

A comparison of the metric of an FLRW spacetime in FLRW coordinates and (the pushforward of) the metric of a strongly spherically symmetric spacetime shows important similarities. Both locally correspond to a warped product metric on $\R \times(0,\infty)$ and a round $d-1$ dimensional unit sphere, either with warping factor $a(t)\,r$ (if $(M,g)$ is an FLRW spacetime and an origin for the spherical coordinates on $\R^d$ has been chosen) or $R$ (if $(M,g)$ is a strongly spherically symmetric spacetime). Therefore, if a coordinate change from one to the other exists, the most natural one would be of the form $T=T(t,r)$, $R=R(t,r)$ and leave the angular part (i.e. the coordinates on the $d-1$ dimensional sphere) unchanged. It was observed in \cite{GallowayLing} that such a coordinate change exists on $M\setminus \{r^2a'(t)^2=1\}$ and is unique (up to an initial condition):

\begin{theorem}[Theorem 4.1 in \cite{GallowayLing}]
\label{theo:unique_spherical}
Let $(M, g)$ be a flat FLRW spacetime in FLRW coordinates $(t,r,\omega)$ where the scale factor $a(t)$
satisfies $a'(0) := \mbox{lim}_{t\rightarrow0^{+}} a'(t) \in (0,\infty]$. Then:

\begin{enumerate}
    \item Subject to a suitable initial condition, there exists a unique transformation of the form $\psi_{s}: (t,r,\omega) \mapsto (T(t,r),R(t,r),\omega)$ such that $g_{s}\coloneqq (\psi_{s})_{\ast}g$ takes the strongly spherically symmetric form:
    
        \begin{equation}\label{eq:strongsphersymmmetric}
        g_{s}= -F(T,R)dT^{2} + G(T,R)dR^{2} + R^{2}d\Omega^{2}_{d-1}    
        \end{equation}
    
    where $F$ and $G$ are regular (away from $\{r^{2}a'(t)^{2}=1\}$ where the Jacobian determinant $J=\frac{\partial (T,R)}{\partial (t,r)}$ vanishes).
    \item Now suppose that $M$ admits a $C^{0}$-extension $\iota: M\rightarrow M_{\textrm{ext}}$. Let $\gamma : (0, 1] \rightarrow M$ be a future directed past inextendible timelike curve parametrized by the coordinate $t$ with past endpoint $\lim_{t\to0^{+}}(\iota\circ\gamma)(t) \in \partial^{-}\iota(M)$, and suppose
    $R$ has a finite positive limit along $\gamma$ as $t\rightarrow 0^{+}$. Then, the following holds: 
    \begin{itemize}
    \item  $\lim_{t\rightarrow 0^{+}} G(\gamma(t)) = 0$. 
    \item If $F(\gamma(t))$ has a finite nonzero limit as $t\rightarrow 0^{+}$, then $\lim_{t\rightarrow0^{+}}T(\gamma(t))=\pm \infty$.
\end{itemize}
\end{enumerate}

\end{theorem}

\begin{remark}\label{rem:details_spherical_proof}
    For the full proof of the previous theorem, see \cite{GallowayLing}. However, let us briefly discuss some consequences of this Theorem and parts of its proof which will also play an important role in the next section. 
    \begin{enumerate}
    \item By starting from an FLRW spacetime in FLRW coordinates we are implicitly excluding not just the hyperplane $\{r^2a'(t)^2=1\}$ but also the timelike line $\{r=0\}$ from $M$. For the future we shall denote $M$ without these two exceptional sets by $\tilde{M}$, i.e.\begin{equation}\label{eq:Mwithoutbadset}\tilde{M}:=M\setminus \big( \{r^2a'(t)^2=1\} \cup \{r=0\}\big).\end{equation}
    \item  While the precise form of $T(t,r)$ and of the metric coefficient $F$ will not be important to us we note the following nice expressions for $R(t,r)$ and the (strongly spherically symmetric) metric coefficient $G$ in terms of the old coordinates $t,r$:
               \begin{equation}\label{eq:R(t,r)}
            R(t,r)=ra(t)
        \end{equation}
        \begin{equation}\label{eq:G_coeff}
            G(t,r)=\frac{1}{1-r^{2}a'(t)^{2}} 
        \end{equation}
        \item   Demanding that the curve $\gamma$ approaching the past boundary in point 2 of the Theorem is defined on $(0,1]$ and parametrized by the FLRW-coordinate $t$ is not restrictive: First any curve in $M$ can be parametrized by $t$ because $t$ is a time function. Second any future directed timelike curve $\gamma :(a,b]\to M$ with $\lim_{s\rightarrow a^{+}}(\iota\circ\gamma)(s)=p\in\partial^{-}\iota(M)$ must satisfy $t(\gamma(s))\to 0$ as $s\to a^+$, since otherwise $t(\gamma(s))\to t_0\in (0,\infty)$ (and $a(t_0)\neq 0$ implies that also the spatial velocity of $\gamma(s) $ must remain bounded as $s\to a^+$ since $\gamma $ is timelike) contradicting $\lim_{s\rightarrow a^{+}}(\iota\circ\gamma)(s)\notin \iota(M)$.
        \item We observe that any timelike curve $\gamma$ parametrized by the $t$-coordinate approaching the past boundary with $\lim_{t\to0^{+}}R(\gamma(t))\in(0,\infty)$ will eventually be contained in $\tilde{M}$, where the strongly spherically symmetric coordinates of Theorem \ref{theo:unique_spherical} are well defined, and $\lim_{t\rightarrow0^{+}} G(\gamma(t))=0$, that is $g_s$ degenerates as one approaches the past boundary:
        Note that as $\lim_{t\to0^{+}}a(t)=0$ and $\lim_{t\to0^{+}}R(\gamma(t))\neq\infty$, it follows that 
        \begin{equation}\label{eq:limr}
            \lim_{t\to0^{+}}r(\gamma(t))= \lim_{t\to0^{+}}\frac{R(\gamma(t))}{a(t)} =\infty    
        \end{equation}
        Hence, as the curve $\gamma$ approaches the past boundary it holds that 
        \begin{equation}
            \lim_{t\to0^{+}}[r^{2}(\gamma(t))a'^{2}(t)]=\infty,
        \end{equation}
        because of the assumption that $\lim_{t\rightarrow0^{+}} a'(t) \neq 0$.
    \end{enumerate}
\end{remark}

\par The importance of the previous theorem will become clearer by defining the following concepts and the follow-up corollary.

\begin{definition}\label{def:natural_str_spher_ext}
   Let $(M,g)$ be a flat FLRW spacetime.
    Let $\iota : (M,g)\rightarrow (M_{\textrm{ext}},g_{\textrm{ext}})$ be a $C^{0}$-extension. We say that $\iota$ is a \emph{natural strongly spherical $C^{0}$-extension} provided that the following holds:
    \begin{enumerate}[label=(\roman*)]
        \item    For all $p\in \partial\iota(M) $ there exists  a neighborhood $V$ of $p$ in $M_{\textrm{ext}}$ and (potentially) a curve $L$ in $V$ such that on $U:=V\setminus L$ there exists a
        diffeomorphism $\psi_{\textrm{ext}}: U\subset M_{\textrm{ext}}\rightarrow \psi_{\textrm{ext}}(U)\subset \mathbb{R} \times (0,\infty) \times \mathbb{S}^2$ such that $g_{\textrm{ext},s}\coloneqq(\psi_{\textrm{ext}})_{\ast}g_{\textrm{ext}}$ takes the strongly spherically symmetric form (\ref{eq:Spher1}) with metric coefficients $F_{\textrm{ext}}, G_{\textrm{ext}}$. In particular, $R_{\textrm{ext}}$ must be in $(0,\infty)$ and $F_{\textrm{ext}}, G_{\textrm{ext}}$ must be nonzero on $U$. 
        \item There exist FLRW coordinates $(t,r,\omega)$ on $M$ such that
        \begin{equation}\label{eq:sphercoordcompatibility}
            \psi_{s}|_{\iota^{-1}(U)\cap \tilde{M}}=\psi_{\textrm{ext}}\circ\iota|_{\iota^{-1}(U)\cap\tilde{M}}
        \end{equation}
        where $\psi_{s}$ are the coordinates from Theorem \ref{theo:unique_spherical} and $\tilde{M}\subseteq M$ is the subset of $M$ where $\psi_s$ is a local diffeomorphism (i.e. $\tilde{M}$ is given by \eqref{eq:Mwithoutbadset}). We may sometimes refer to $\psi_s$ as the {\em natural strongly spherical change of coordinates}. 
        Note that the previous expression implies that on $\psi_s(\iota^{-1}(U)\cap \tilde{M})$
        \begin{equation}\label{eq:sphermetriccompatibility}
         g_{\textrm{ext},s}=(\psi_{\textrm{ext}})_{\ast}g_{\textrm{ext}}=(\psi_{\textrm{ext}})_{\ast}\iota_{\ast}g=(\psi_{s})_{\ast}g= g_{s}
        \end{equation}
       for $g_s$ from \eqref{eq:strongsphersymmmetric}. 
    \end{enumerate} 
\end{definition}

According to Theorem \ref{theo:unique_spherical} the considered class of FLRW metrics in strongly spherically symmetric coordinates becomes degenerate (as $G\rightarrow0$) when approaching the past boundary. This property can be used to prove that, in fact, such a  FLRW spacetime has no natural strongly spherical extension.

\begin{corollary}[Corollary 4.2 in \cite{GallowayLing}]
\label{nosphericalextension}
Let $(M,g)$ be a future inextendible flat FLRW spacetime with $a'(0) \in (0,\infty]$. Then there exists no natural strongly spherically symmetric $C^{0}$-extension of $(M,g)$.
\end{corollary}

\begin{proof}
The proof is by contradiction. Let $(M,g)$ be a flat FLRW spacetime
satisfying that $a'(0)\in (0,\infty]$. 
Assume there exists a natural strongly spherical $C^{0}$-extension $\iota : (M,g)\rightarrow (M_{\textrm{ext}},g_{\textrm{ext}})$ and let $p\in\partial^{-}\iota(M)$. By definition there exists a neighbourhood $V$ of $p$ in $M_{\textrm{ext}} $, a curve $L$ in $V$ and a 
diffeomorphism $\psi_{\textrm{ext}}: U:=V\setminus L \subset M_{\textrm{ext}} \rightarrow \psi_{\textrm{ext}}(U)\subset \mathbb{R} \times (0,\infty) \times \mathbb{S}^{d-1}$ such that, in $U$, $g_{\textrm{ext},s}\coloneqq(\psi_{\textrm{ext}})_{\ast}g_{\textrm{ext}}$ can be written in the strongly spherically symmetric form
\begin{equation}
   g_{\textrm{ext},s}=-F_{\textrm{ext}}(T_{\textrm{ext}},R_{\textrm{ext}})dT_{\textrm{ext}}^{2}+G_{\textrm{ext}}(T_{\textrm{ext}},R_{\textrm{ext}})dR_{\textrm{ext}}^{2}+R_{\textrm{ext}}^{2}d\Omega^{2}_{\textrm{ext},d-1} .
\end{equation}
Equations \eqref{eq:sphercoordcompatibility} and \eqref{eq:sphermetriccompatibility} imply that $(T_{\textrm{ext}}\circ \iota,R_{\textrm{ext}}\circ \iota,\omega_{\textrm{ext}}\circ \iota)$  have to agree with $(T ,R,\omega)$ in $\iota^{-1}(U)\cap \tilde{M}$ and that  $$(F_{\textrm{ext}},G_{\textrm{ext}})=(F,G) \quad\quad \mathrm{on}\quad \; \psi_{\textrm{ext}}(\iota(\tilde{M})\cap U)= \psi_s(\iota^{-1}(U)\cap \tilde{M}).$$
\par Without loss of generality we may assume that $p\in U$: If $p$ were not itself in $U$ we could replace it by a point $q\in U\cap \partial^{-}\iota(M)$. Note that such a  $q$ must exist for dimensional reasons since $\partial^{-}\iota(M)$ is a topological hypersurface, cf.\ Theorem \ref{FLRWfutureboundary}, and $U=V\setminus L$ where $V$ is an open neighbourhood of $p\in \partial^{-}\iota(M)$ and $L$ is at most a 
curve.

Let $\gamma : [0, 1] \rightarrow U\subseteq M_{\textrm{ext}}$ be a future directed timelike curve with $\gamma((0,1])\subset \iota(M)$ and $\gamma(0)=p$ (this curve exists by definition of $\partial^{-}\iota(M)$) parametrized by the time function $t$. In particular, $\lim_{s\rightarrow0}R(\gamma(s))\neq0$ as $R_{\textrm{ext}}\in(0,\infty)$ (by definition of the strongly spherical coordinates) and $R$ and $R_{\textrm{ext}}$ agree in $\iota(M)\cap U$ (so $\lim_{s\rightarrow0}R(\gamma(s))=\lim_{s\rightarrow0}R_{\textrm{ext}}(\gamma(s))=R_{\textrm{ext}}(p)$). 
Thus, by the second part of Theorem \ref{theo:unique_spherical} we have 
$\lim_{t\rightarrow 0^{+}} G(\iota^{-1}(\gamma(t))) = 0$. Further, 
$\gamma(t)$ is contained in $\iota(\tilde{M})$ for all $t>0$ close enough to $0$
by point 4 of Remark \ref{rem:details_spherical_proof}. Hence, 
also $G_{\textrm{ext}}(p)=\lim_{t\rightarrow 0^{+}} G_{\textrm{ext}}(\gamma(t))=0$. Therefore, $g_{\textrm{ext},s}$ 
is degenerate at $p$, contradicting $g_{\textrm{ext},s}$ being the extended metric in strongly spherically symmetric coordinates near $p$. So no natural strongly spherical $C^{0}$-extension can exist. \end{proof}

Note that this corollary is not quite stating that, given a flat FLRW spacetime with $a'(0) \in (0,\infty]$, there does not exist an extension $\iota : (M,g)\rightarrow (M_{\textrm{ext}},g_{\textrm{ext}})$ such that $(M_{\textrm{ext}},g_{\textrm{ext}})$ is strongly spherically symmetric. It shows that there is no strongly spherical extension that is compatible with the spherical symmetry of the original FLRW spacetime in the way described. This is certainly acknowledged in \cite{GallowayLing}, but we decided to introduce the terminology of \textit{natural} strongly spherical extension to make it even more explicit. 

\begin{remark}\label{rem:nospherical_hyperbolic}
For the sake of conciseness we do not present it in detail in this paper, but in \cite{GallowayLing} Galloway and Ling also proved a theorem and corollary analogous to Theorem \ref{theo:unique_spherical} and Corollary \ref{nosphericalextension} for a certain type of hyperbolic FLRW spacetimes (see also Corollary \ref{cor:nospherical_hyp_FLRW}). Let $(M,g)$ be a future inextendible hyperbolic FLRW spacetime, $\iota : M\rightarrow M_{\textrm{ext}}$ an arbitrary $C^{0}$-extension and $\gamma : (0,1]\rightarrow M$ a future directed past inextendible timelike curve parametrized by the coordinate $t$ with $\lim_{t\to0^{+}}(\iota\circ\gamma)(t) \in \partial^{-}\iota(M)$. Ling and Galloway showed that if $a'(0)\in[0,\infty]$ and $a'(0) \neq 1$, then there exists a unique (subject to an initial condition and away from a subset of the spacetime where the determinant of the Jacobian of the strongly spherically symmetric change of coordinates vanishes) strongly spherically symmetric change of coordinates that leaves the spherical coordinate $\omega$ invariant, which furthermore satisfies that the metric coefficient $G$ vanishes as $t\to0^{+}$ provided that $\lim_{t\to0^{+}}(R(\gamma(t)))\in(0,\infty)$. Note that again $R(t,r)=ra(t)$ and the metric coefficient $G$ is given by (see \cite[Theorem 4.3 and equation (4.24)]{GallowayLing}):
\begin{equation}
    G(t,r)=\frac{1}{\cosh^2(r)-\sinh^2(r)a'(t)^2}=\frac{a^{2}(t)}{R(t,r)^{2}(1-a'^{2}(t))+a^{2}(t)}
\end{equation}
 From the previous expression it becomes clear that it is important to impose that $a'(0)\neq1$ and $\lim_{t\to0^{+}}(R(\gamma(t)))\neq0$ so that it holds that $\lim_{t\to0^{+}}G(\gamma(t))=0$. We will come back to this remark in Corollary \ref{cor:noaxisymmetric_Hyp_FLRW}.
\end{remark}

\section{Strongly axisymmetric inextendibility}\label{sec:cylindrical}

In this section we will restrict ourselves to 4-dimensional spacetimes. The aim is to generalize the previously discussed results to a subclass of axisymmetric spacetimes. As in the previous section (when considering \emph{strongly} spherically symmetric spacetimes), we will focus on a subclass of axisymmetric spacetimes for which the metric takes the following form 
$$g=-AdT^{2}+Bdz^{2}+Cd\rho^{2}+\rho^{2}d\varphi^{2}+\left(D_{1}dT+D_{2}dz+D_{3}d\rho \right)d\varphi$$
with coordinates $(T,Z,\rho,\varphi)$ and smooth functions $A,B,C$ which are allowed to depend on $T,z$ and $\rho$. We will call these spacetimes \emph{strongly} axisymmetric spacetimes. Some specific examples of spacetimes whose metric takes such a form (but which do not necessarily fit into the FLRW framework) are: 1) the Gödel Universe (for which $D_{2}=D_{3}=0$, cf. \cite[Section 2.10.1]{catalogue}), 2) cylindrically symmetric (in the sense of \cite[Section 22.1]{exact_solutions}) spacetimes, where the metric coefficients are also independent of the $z$ coordinate and $D_{1}=D_{3}=0$, or 3) the Weyl class solutions to the Einstein equations (cf. \cite[Section 20.2]{exact_solutions}), where the metric coefficients are time independent and $D_{2}=D_{3}=0$.

\begin{definition}[Strongly axisymmetric spacetimes]\label{def:axisym}
    Let $(M,g)$ be a 4-dimensional spacetime. It is called a \emph{strongly axisymmetric spacetime} if for all $p\in M$ there exists a neighbourhood $V$ of $p$ in $M$ and (potentially) a two dimensional submanifold $S\subset V$ such that for $U:=V\setminus S$ there exists a diffeomorphism $\psi_{a}=(T,z,\rho,\varphi):U\to \psi_{a}(U)\subset \R^2 \times (0,\infty)\times \mathbb{S}^1$ with $T: U\to \R$, $z: U\rightarrow \mathbb{R}$, $\rho : U\rightarrow(0,\infty)$,  and $\varphi: U\rightarrow \mathbb{S}^1$
    such that in the coordinate neighborhood the pushforward of the metric takes the form:
\begin{equation}
\label{eq:Cylind1}
    (\psi_{a})_{\ast}g=-AdT^{2}+Bdz^{2}+Cd\rho^{2}+\rho^{2}d\varphi^{2}+\left(D_{1}dT+D_{2}dz+D_{3}d\rho \right)d\varphi
\end{equation}
where $d\varphi^{2}$ is the standard metric on $\mathbb{S}^1$ and the metric coefficients are smooth functions of the coordinates $T$, $z$ and $\rho$ only. We call $(T,z,\rho, \varphi)$ \emph{strongly axisymmetric coordinates}. 
\end{definition}
\begin{remark}
    Note that in the previous definition no assumption on the sign of the metric coefficients is being made, but we of course require the signature to remain Lorentzian. 
\end{remark}
\par As one would expect strongly spherically symmetric spacetimes are also strongly axisymmetric. To motivate the changes of coordinates we are considering let us briefly consider the 3-dimensional Euclidean space $(\mathbb{R}^{3},\delta)$ with spherical coordinates $(R,\theta,\varphi)$. The metric $\delta$ can then be expressed in cylindrical coordinates $(z,\rho,\varphi)$ through the canonical change of coordinates $\rho(R,\theta)=R\sin{\theta}, z(R,\theta)=R\cos{\theta}$, which leaves $\varphi$ fixed. In the following theorem we show that this type of ``natural" coordinate change generalizes to an arbitrary strongly spherically symmetric spacetime $(M,g)$ and is essentially unique. Note that while we formulate the theorem for global  spherically symmetric 
coordinates we can of course apply it to subsets $U\subseteq M$ as well.

\begin{theorem}
\label{uniquecylindrical}
    Let $(M,g)$ be a four dimensional strongly spherically symmetric  spacetime with strongly spherically symmetric coordinates $\psi_{s}:M\to \psi_{s}(M)\subseteq \R\times (0,\infty)\times \mathbb{S}^2$ and with $g_{s}=(\psi_{s})_{\ast}g$ the metric in strongly spherically symmetric coordinates $(T,R,\omega)$. Then, subject to a suitable initial condition and the choice of coordinates $\theta$ and $\varphi$ on $\mathbb{S}^2\setminus \{\mathrm{pt.},-\mathrm{pt.}\} \cong (0,\pi)\times \mathbb{S}^1$ (i.e., the choice of axis for the axial symmetry), there exists a unique {local diffeomorphism} of the form $\psi_{a} : (R,\theta)\mapsto (z,\rho)$ such that $g_{a}\coloneqq(\psi_{a})_{\ast}(\psi_{s})_{\ast}g$ is of the strongly axisymmetric form $(\ref{eq:Cylind1})$. Moreover, the metric coefficients $D_1$ to $D_3$ vanish and $A,B,C$ are regular (away from certain measure zero sets on which the change of coordinates is not well defined). 
\end{theorem}

\begin{proof}
In order to prove this theorem we closely follow the strategy used in \cite{GallowayLing} to prove Theorem \ref{theo:unique_spherical}: construct an explicit change of coordinates from a strongly spherically symmetric spacetime to a strongly axisymmetric spacetime. As an ansatz, we suppose there exists a smooth and invertible change of coordinates $\psi_{a}$ of the following form:
\begin{equation}
\label{eq:SAS2}
\begin{cases}
    z=z(R,\theta) \longrightarrow dz^{2}=z_{R}^{2}dR^{2}+z_{\theta}^{2}d\theta^{2}+2z_{R}z_{\theta}dRd\theta\\
    \rho=\rho(R,\theta) \longrightarrow d\rho^{2}=\rho_{R}^{2}dR^{2}+\rho_{\theta}^{2}d\theta^{2}+2\rho_{R}\rho_{\theta}dRd\theta\\
\end{cases}
\end{equation}
such that $g_{a}\coloneqq(\psi_{a})_{\ast}(\psi_{s})_{\ast}g$ can be written as in (\ref{eq:Cylind1}) and where $\psi_{s}$ is are the strongly spherically symmetric coordinates. Note that $z_{R}=\frac{\partial z}{\partial R}$ and $z_{\theta}=\frac{\partial z}{\partial \theta}$ (analogous for $\rho_{R}$ and $\rho_{\theta}$). As the considered change of coordinates does not affect the $T$ and $\varphi$ coordinates it directly follows that $D_{1}=D_{2}=D_{3}=0$. Hence, under the change of coordinates (\ref{eq:SAS2}), it holds that $(\psi_{a})_{\ast}(\psi_{s})_{\ast}g=(\psi_{s})_{\ast}g$:
\begin{equation}
\label{eq:SAS3}
    -FdT^{2}+GdR^{2}+R^{2}(d\theta^{2}+\sin^{2}\theta d\varphi^{2})=-AdT^{2}+Bdz^{2}+Cd\rho^{2}+\rho^{2}d\varphi^{2}
\end{equation}
For now we do not specify if $A$, $B$ and $C$ are positive or negative (it will be discussed in the second half of the proof, once the explicit form of the change of coordinates is obtained). From the previous expression we can make two quick conclusions:
\begin{itemize}
    \item As the change of coordinates does not affect the coordinate $T$, it is clear that $A=F$.
    \item Leaving $T, R$ and $\theta$ constant we see that $R^{2}\sin^{2}\theta d\varphi^{2}=\rho^{2}d\varphi^{2}$. Thus, $\rho(R,\theta)=R\sin\theta$.
\end{itemize}
Moreover, replacing (\ref{eq:SAS2}) in the right hand side of (\ref{eq:SAS3}), we get 3 new independent equations:
\begin{subequations}
    \begin{alignat}{3}
        G dR^{2}= (B z^{2}_{R}+C\rho^{2}_{R})dR^{2} \hspace{0.5cm} & \longrightarrow \hspace{0.5cm} B z^{2}_{R}=G-C\rho^{2}_{R}\label{eq:SAS3subeq1} \\
        R^{2}d\theta^{2}=(Bz^{2}_{\theta}+C\rho^{2}_{\theta})d\theta^{2} \hspace{0.5cm} & \longrightarrow \hspace{0.5cm} Bz^{2}_{\theta}=R^{2}-C\rho^{2}_{\theta} \label{eq:SAS3subeq2}\\
        0 = (Bz_{R}z_{\theta}+C\rho_{R}\rho_{\theta})dRd\theta \hspace{0.5cm} & \longrightarrow \hspace{0.5cm} Bz_{R}z_{\theta}=-C\rho_{R}\rho_{\theta} \label{eq:SAS3subeq3}
    \end{alignat}
\end{subequations}
The following calculations hold as long as $G\cos^{2}\theta+\sin^{2}\theta\neq 0$ and $z_{\theta}\neq 0$. The first corresponds to the three dimensional submanifold $G(T,R)=-\tan^2(\theta)$, so is in particular a measure zero set, and we will later discuss that the second cannot happen.
\par If we square (\ref{eq:SAS3subeq3}), plug (\ref{eq:SAS3subeq1}) and (\ref{eq:SAS3subeq2}) in it and use that $\rho_{\theta}=R \cos \theta$ and $\rho_{R}=\sin \theta$ we get an explicit expression for C:
\begin{align}
\label{eq:SAS4}
    & B^{2}z^{2}_{R}z^{2}_{\theta}=C^{2}\rho^{2}_{R}\rho^{2}_{\theta} \nonumber\\
    &\iff (G-C\rho^{2}_{R})(R^{2}-C\rho^{2}_{\theta}) = C^{2}\rho^{2}_{R}\rho^{2}_{\theta}\nonumber\\
    &\iff C=\frac{GR^{2}}{G\rho^{2}_{\theta}+R^{2}\rho^{2}_{R}}=\frac{G}{G\cos^{2}\theta+\sin^{2}\theta}
\end{align}
Note that from ($\ref{eq:SAS3subeq2}$) we have that $B=\frac{R^{2}(1-C\cos^{2}\theta)}{z^{2}_{\theta}}$. If we replace this equation and the expression for $C$ in (\ref{eq:SAS3subeq3}) we obtain:
\begin{equation}
\label{eq:SAS5}
    \frac{z_{R}}{z_{\theta}}=-\frac{G\cos \theta}{R \sin \theta}
\end{equation}
This linear PDE corresponds to a transport equation with variable coefficients. In order to solve it, the method of characteristics is used so that the PDE becomes an ODE along a specific type of curves. In particular we choose a curve $\theta(R)$ in the $(R,\theta)$-plane that satisfies $\frac{d}{dr}z(R,\theta(R))=0$ and thus by the chain rule:
\begin{equation*}
    0\overset{!}{=}\frac{d}{dR}z(R,\theta(R))=z_{R}+z_{\theta}\frac{d \theta}{dR}
\end{equation*}
Combining this with (\ref{eq:SAS5}) and restricting to $\theta\neq \tfrac{\pi}{2}$ we get the characteristic equation of the considered curves:
\begin{equation}\label{eq:characteristic_eq}
    \frac{d\theta}{dR}=-\frac{z_{R}}{z_{\theta}}=\frac{G\cos \theta}{R \sin \theta} \hspace{0.5cm}\longrightarrow \hspace{0.5cm} \tan \theta d\theta= G\frac{dR}{R}
\end{equation}
Integrating on both sides (using the substitution $u=\cos{\theta}$ for the integral on the left hand side), we get the curve: 

\begin{equation*}
    -\ln|\cos\theta|= \int\frac{G(T,R)}{R}dR
\end{equation*} 
These are the characteristic curves along which the solutions of the PDE (\ref{eq:SAS5}) are constant. Therefore, a general solution to (\ref{eq:SAS5}) is of the form:
\begin{align}
    & z(R,\theta)=
    \begin{cases}
    f_{1}\left(\int\frac{G(T,R)}{R}dR+\ln|\cos\theta|\right) \hspace{1cm}\textrm{for}\; \theta\in (0,\frac{\pi}{2})\\
    f_{2}\left(\int\frac{G(T,R)}{R}dR+\ln|\cos\theta|\right) \hspace{1cm}\textrm{for}\; \theta\in (\frac{\pi}{2},\pi)\\
    \end{cases}
\end{align}
where $f_{1}, f_{2} : \mathbb{R}\rightarrow \mathbb{R}$ are arbitrary smooth functions, which correspond to the initial condition we can prescribe for our change of coordinates. Note that the reason that two arbitrary smooth functions $f_{1}$ and $f_{2}$ are obtained is due to the fact that, when integrating expression \eqref{eq:characteristic_eq}, the intervals $(0,\frac{\pi}{2})$ and $(\frac{\pi}{2},\pi)$ have to be considered separately (as $\theta \neq \pi/2$). We use the freedom in choosing an initial condition for our change of coordinates to 'match' the functions $f_{1}$ and $f_{2}$ and make the following formulas nicer. Given an arbitrary smooth function $f : \mathbb{R} \rightarrow\mathbb{R}$ with $f'\neq0$ at each point (afterwards we will show that $f'=0$ actually leads to a contradiction with \eqref{eq:SAS3subeq1}-\eqref{eq:SAS3subeq3}; note that by \eqref{eq:ztheta} $f'=0$ if and only if $z_{\theta}=0$, where our derivation was not valid) and $f(0)=0$, the functions $f_{1}$ and $f_{2}$ are fixed by demanding that $f_{1}=f\circ\exp$ and $f_{2}=f\circ(-\exp)$. Then the expression for the coordinate $z(R,\theta)$ simplifies to
\begin{equation}
    z(R,\theta)=f(\cos{\theta}e^{\int\frac{G(T,R)}{R}dR}),
\end{equation}
which is continuous across $\theta=\frac{\pi}{2}$. Note that 
\begin{equation}\label{eq:ztheta}
    z_{\theta}=-\sin{\theta}e^{\int\frac{G(T,R)}{R}dR}\,f'
\end{equation} which is always non-zero for our choice of $f$. Replacing this and expression (\ref{eq:SAS4}) for $C$ in (\ref{eq:SAS3subeq2}) we get that $B$ is:
\begin{align}
    B&=\frac{R^{2}(1-C\cos^{2}\theta)}{z^{2}_{\theta}}=\frac{R^{2}}{e^{2\int\frac{G(T,R)}{R}dR}f'^{2}(G\cos^{2}\theta+\sin^{2}\theta)}
\end{align}
Altogether, away from $\{G=-\tan^2(\theta)\}$ we have found the following change of coordinates $\psi_{a} : (R,\theta)\mapsto (z,\rho)$ from a strongly spherically symmetric spacetime to a strongly axisymmetric spacetime, where:
\begin{equation}
\label{eq:SAS6}
\begin{cases}
    z(R,\theta) =f\left(\cos{\theta}e^{\int\frac{G(T,R)}{R}dR}\right) \\
    \rho(R,\theta) = R\sin\theta\\
\end{cases}
\end{equation}
Under this change of coordinates, the metric coefficients of $g_{c}$ are:
\begin{equation}
\label{eq:SAS7}
\begin{cases}
    A(T,z,\rho)=F(T,R) \\
    B(T,z,\rho)= \frac{R^{2}}{e^{2\int\frac{G(T,R)}{R}dR}f'^{2}(G\cos^{2}\theta+\sin^{2}\theta)} \\
    C(T,z,\rho)= \frac{G}{G\cos^{2}\theta+\sin^{2}\theta}
\end{cases}
\end{equation}

\par In particular, the coefficients $A$, $B$ and $C$ are positive or negative depending on the value of $F$ and $G$. Recall that, by definition, for strongly spherically symmetric spacetimes the metric coefficients $F$ and $G$ are either both positive or both negative:
\begin{enumerate}[label=(\roman*)]
    \item \underline{$F>0$ and $G>0$}: It is clear that in this case $A>0$, $B>0$ and $C>0$.
    \item \underline{$F<0$ and $G<-\tan^{2}{\theta}$}: From the expressions \eqref{eq:SAS7} it follows that $A<0$, $B<0$ and $C>0$. 
    \item \underline{$F<0$ and $-\tan^{2}{\theta}<G<0$}: Finally, in this case, it follows that $A<0$, $B>0$ and $C<0$.
\end{enumerate}

In summary, we can write the strongly spherically symmetric metric in an explicit strongly axisymmetric form as follows:
\begin{align}
\label{eq:SAS8}
    g_{a} &\coloneqq(\psi_{a})_{\ast}(\psi_{s})_{\ast}g= -A(T,z,\rho)dT^{2}+B(T,z,\rho)dz^{2}+C(T,z,\rho)d\rho^{2}+\rho^{2}d\varphi^{2} \nonumber\\
    &= -FdT^{2}+\frac{1}{G\cos^{2}\theta+\sin^{2}\theta}\left(\frac{R^{2}}{e^{2\int\frac{G(T,R)}{R}dR}f'^{2}}dz^{2}+Gd\rho^{2}\right)+\rho^{2}d\varphi^{2}
\end{align}
where $\psi_{s}$ is again the change of coordinates such that $(\psi_{s})_{\ast}g$ can be written in strongly spherically symmetric coordinates. In order for $\psi_{a} : (R,\theta)\mapsto (z,\rho)$ to be a local diffeomorphism, it has to be smooth and locally invertible. By the inverse function theorem we only have to check if the determinant of the Jacobian of this change of coordinates (\ref{eq:SAS6}) vanishes or not. We obtain
\begin{equation}
    J=z_{R}\rho_{\theta}-z_{\theta}\rho_{R}=z_{\theta}\left(-\frac{G\cos{\theta}}{R\sin{\theta}}\rho_{\theta}-\rho_{R}\right)=f'e^{\int\frac{G(T,R)}{R}dR}( G\cos^{2}\theta+\sin^{2}\theta)
\end{equation}
which is non-zero away from $G=-\tan^{2}\theta$, where $J$ vanishes.

\par Finally, let us for completeness discuss why it isn't actually necessary to assume $f'\neq0$ (or equivalently $z_{\theta}\neq 0$):
Plugging $z_{\theta}=0$ into equations (\ref{eq:SAS3subeq2}) and (\ref{eq:SAS3subeq3}) with $\rho=R\sin{\theta}$ we obtain
\begin{equation}
    \label{eq:tbd}
    \begin{cases}
        1 = C\cos^{2}{\theta}\\
        0 = CR\cos{\theta}\sin{\theta}
    \end{cases}    
    \end{equation}
    However these equations cannot hold simultaneously as $C,R$ and $\sin(\theta)$ are nonzero (as we only consider $\theta \in (0,\pi)$ and demanded $g_a$ to be non-degenerate), so the first implies $\cos(\theta)\neq 0$ while the second implies $\cos(\theta)=0$.

    Hence, we have shown that, subject to an initial condition and the choice of coordinates $\theta$ and $\varphi$ on $\mathbb{S}^2$ in the strongly spherically symmetric coordinates, there exists a unique transformation of coordinates $\psi_{a}$ of the form $z=z(R,\theta), \rho=\rho(R,\theta)$ which is well defined and a local diffeomorphism almost everywhere on $M$, more specifically away from the subsets $\{\theta=0 \}\cup \{\theta=\pi\}$, where the chosen coordinates on $\mathbb{S}^2$ break down, and $\{G=-\tan^{2}{\theta}\}$, where the change of coordinates breaks.
\end{proof}

Theorem \ref{uniquecylindrical} states that all strongly spherically symmetric spacetimes have a unique (subject to an initial condition and the choice of $\theta$ and $\varphi$) natural strongly axisymmetric change of coordinates. Hence, it gives a uniqueness result within the class of changes of coordinates of the form $\psi_{a}: (R,\theta) \mapsto (z,\rho)$. Similar to how Galloway and Ling obtained Corollary \ref{nosphericalextension} from Theorem \ref{theo:unique_spherical} we want to use Theorem \ref{uniquecylindrical} to show that certain strongly axisymmetric spacetimes have no natural strongly axisymmetric $C^0$-extension. In order to do so, let us define the following concepts.

\begin{definition}
\label{naturalcylindrical}
    Let $(M,g)$ be a strongly spherically symmetric spacetime of dimension $4$ together with strongly spherically symmetric coordinates $\psi_s$ defined on some $\tilde{M}\subseteq M$. Let $\iota : (M,g)\rightarrow (M_{\textrm{ext}},g_{\textrm{ext}})$ be a $C^{0}$-extension. We say that $\iota$ is a  \emph{natural strongly axisymmetric ${C^{0}}$-extension} compatible with $\psi_s$ provided that the following holds:   
        \begin{enumerate}[label=(\roman*)]
        \item 
       For all $p\in \partial \iota(M)$ there exists a neighborhood $V$ of $p$ in $M_{\textrm{ext}}$ and (potentially) a two dimensional submanifold $S\subset V$ such that $U:=V\setminus S$ satisfies  $$\iota^{-1}(U)\cap \big( M\setminus \tilde{M}\big) =\emptyset$$ and such that there exists a diffeomorphism $\psi_{\textrm{ext,a}} : U\subseteq M_{\textrm{ext}} \rightarrow \psi_{\textrm{ext,a}}(U) \subseteq \R\times \R \times (0,\infty) \times \mathbb{S}^1$ such that $g_{\textrm{ext},a}\coloneqq(\psi_{\textrm{ext,a}})_{\ast}g_{\textrm{ext}}$ takes the strongly axisymmetric form (\ref{eq:Cylind1}) with metric coefficients $A_{\textrm{ext}},B_{\textrm{ext}},C_{\textrm{ext}},(D_1)_{\textrm{ext}}, (D_2)_{\textrm{ext}}$ and $(D_3)_{\textrm{ext}}$. In particular  $g_{\textrm{ext},a}$ must be non-degenerate on $\psi_{\textrm{ext,a}}(U)$. 
        \item We have 
        \begin{equation}\label{eq:natural_strongly_axi}
        \psi_{a}\circ\psi_{s}|_{\tilde{U}}=\psi_{\textrm{ext,a}}\circ\iota|_{\tilde{U}}
        \end{equation}
        where $\psi_{a}$ is the change from spherically symmetric coordinates $(T,R,\theta,\varphi)$ to strongly axisymmetric coordinates from Theorem \ref{uniquecylindrical} and $\tilde{U}\subseteq \iota^{-1}(U) \subseteq \tilde{M}$ is the subset of $\iota^{-1}(U)$ where $\psi_a$ is a local diffeomorphism, that is
        $$ \tilde{U}= \iota^{-1}(U)\setminus \big( \{G=-\tan^2\theta \}\cup\{\theta=0,\pi\} \big). $$

        We may sometimes refer to $\psi_a$ as a \emph{natural strongly axisymmetric change of coordinates}.
        
        Recall that $g_{\textrm{ext}|_{\iota(M)}}=\iota_{\ast}g$. So the previous expression implies that on $\psi_{s}(\tilde{U})$
        \begin{equation}
            g_{\textrm{ext},a} 
            =(\psi_{\textrm{ext,a}})_{\ast}g_{\textrm{ext}}
            =(\psi_{\textrm{ext,a}})_{\ast}\iota_{\ast}g=(\psi_{a})_{\ast}(\psi_{s})_{\ast}g= g_{a}.
        \end{equation}
        \end{enumerate}
        
\end{definition}

\par This leads us directly to the next theorem which, essentially, gives conditions under which there is no natural strongly axisymmetric extension of $(M,g)$.

\begin{theorem}
\label{nocylindrical}
    Let $(M,g)$ be a future inextendible four dimensional strongly spherically symmetric spacetime with with strongly spherically symmetric coordinates $\psi_s$ defined on some $\tilde{M}\subseteq M$.
    Let $\iota : (M,g)\rightarrow (M_{\textrm{ext}},g_{\textrm{ext}})$ be a $C^{0}$-extension. Let $\gamma : (0,1]\rightarrow M$ be a past inextendible curve with $\lim_{s\to0}(\iota\circ\gamma)(s)=p\in\partial^{-} \iota(M)$. If $\gamma((0,1])\subset \tilde{M}$ and $\lim_{s\rightarrow0}G(\gamma(s))=0$ and $\lim_{s\rightarrow0^{+}}R(\gamma(s))\in (0,\infty)$, then either $\iota $ cannot be a natural strongly axisymmetric extension of $(M,g)$ compatible with $\psi_s$ or $p$ lies on the axis of symmetry, i.e., $p\in S$.
\end{theorem}

\begin{proof} Let $\iota : (M,g)\rightarrow (M_{\textrm{ext}},g_{\textrm{ext}})$ be an arbitrary $C^{0}$-extension and $\gamma : (0,1]\rightarrow \tilde{M}$ a past inextendible curve with $\lim_{s\to0^{+}}(\iota\circ\gamma)(t)=p\in\partial^{-} \iota(M)$ and such that $\lim_{s\rightarrow0^{+}}G(\gamma(s))=0$ and $\lim_{s\rightarrow0^{+}}R(\gamma(s))\in (0,\infty)$. 
Assume that $\iota $ is a a natural strongly axisymmetric extension of $(M,g)$ compatible with the given realization of the strong spherical symmetry and $p\notin S$.
Then there exists a neighbourhood $U\subseteq M_{\textrm{ext}}$ of $p$ and a diffeomorphism $\psi_{\textrm{ext},a}:U\to \psi_{\textrm{ext},a}(U)\subseteq \R^2\times (0,\infty)\times \mathbb{S}^1$ such that $\iota^{-1}(U)\subseteq \tilde{M}$ and \eqref{eq:natural_strongly_axi} holds for $\psi_a$ from Theorem \ref{uniquecylindrical}, i.e.
$$  \psi_{a}\circ\psi_{s}|_{\tilde{U}}=\psi_{\textrm{ext,a}}\circ\iota|_{\tilde{U}}.$$
We observe that the curve $\gamma$ eventually lies in $\tilde{U}$ as it does not intersect the region $\{G=-\tan^{2}{\theta}\}$ where the strongly axisymmetric change of coordinates is not well-defined: if $\gamma$ were to intersect $\{G=-\tan^{2}{\theta}\}$ for arbitrarily small $s$, then as $\lim_{s\to0^{+}}G(\gamma(s))=0$ it should also hold that $\lim_{s\to0^{+}}\theta(\gamma(s))\in\{0,\pi\}$. But then $\rho(\gamma(s))=R(\gamma(s)) \sin(\theta(\gamma(s))) \to 0$ (as by assumption $\lim_{s\to0^{+}}R(\gamma(s))\in (0,\infty) $)  which contradicts that $\rho(p) \in (0,\infty)$ since $p\notin S$. Note that this is an analogous argument to the one appearing in Remark \ref{rem:details_spherical_proof} which shows that $\gamma$ does not intersect the problematic region $\{r^{2}a'^{2}(t)\}$.

Now recall the explicit expression of the metric coefficient $C$ in terms of the strongly spherically symmetric coordinates:
\begin{equation*}
    C(T,z,\rho)= \frac{G}{G\cos^{2}\theta+\sin^{2}\theta}\nonumber
\end{equation*}
on $\tilde{U}$.
 Then, it directly follows that
\begin{equation*}
    \lim_{s\to0^{+}}C(\gamma(s))=0,    
\end{equation*}
i.e. the metric $g_{a}=(\psi_{a})_{\ast}g_{s}=g_{\textrm{ext},a}$ degenerates at $p$, a contradiction.
\end{proof}

\par Analogously to the discussion after Corollary \ref{nosphericalextension}, the previous theorem does not state that a strongly spherically symmetric spacetime cannot have a strongly axisymmetric extension. It states that there is no \textit{natural} strongly axisymmetric extension of a strongly spherically symmetric spacetime compatible with some given strongly spherically symmetric coordinates. As before, the uniqueness argument used in the proof of the theorem only holds for the class of changes of coordinates of the class $\psi_{a}: (R,\theta) \mapsto (z,\rho)$. Therefore, in principle, there could also exist a change of coordinates of the form $\overline{\psi}_{a} : (T,R,\theta,\varphi)\mapsto (\overline{T},z,\rho,\overline{\varphi})$ such that the pushforward of the extended metric $g_{\textrm{ext}}$ could be written in the strongly axisymmetric form (\ref{eq:Cylind1}).
\par The previous theorem can be used to state an inextendibility result for FLRW spacetimes similar to Corollary \ref{nosphericalextension}. Recall that, if $(M,g)$ is an flat FLRW spacetime with $a'(0)\in (0,\infty]$, then there exists a unique (up to an initial condition) natural strongly spherical change of coordinates $\psi_{s}: (t,r) \mapsto (T,R)$. In order to find a change of coordinates from the flat FLRW spacetime to a strongly axisymmetric spacetime, one could either directly look for a change of coordinates $\overline{\psi} : (t,r,\theta,\varphi)\mapsto(T,z,\rho,\varphi)$ or use the natural strongly spherical change of coordinates $\psi_{s}$ from Theorem \ref{theo:unique_spherical} as part of a two-step transformation:
\begin{equation}
\label{eq:SAS15}
    \begin{split}
    \{t,r,\theta,\varphi\}
    \end{split}
\quad\xrightarrow{\psi_{s}}\quad
  \begin{split}
    \begin{cases}
        T=T(t,r)\\
        R=R(t,r) \\
        $\textrm{$\theta, \varphi$ unchanged}$
    \end{cases}
  \end{split}
\quad\xrightarrow{\psi_{a}}\quad
  \begin{split}
        \begin{cases}
            z=z(R,\theta)\\
            \rho=\rho(R,\theta)\\
            $\textrm{$T, \varphi$ unchanged}$
        \end{cases}
    \end{split}
\end{equation}
We note that, if $\iota:M\to M_{\textrm{ext}}$ is any $C^0$-extension, then the coefficient $G$ of $(\psi_s)_{\ast}g$ satisfies $\lim_{t\to0^{+}}G(\gamma(t))=0$ along any timelike curve $\gamma:(0,1]\to M$ with $\lim_{t\to0^{+}}(\iota\circ\gamma)(t)\in\partial^{-}\iota(M)$  and $\lim_{t\to0^{+}}R(\gamma(t))\in (0,\infty)$ by the second part of Theorem \ref{theo:unique_spherical}. Thanks to this we obtain the following Corollary from Theorem \ref{nocylindrical}.

\begin{corollary}\label{cor:noaxisymmetric_Flat_FLRW}
    Let $(M,g)$ be a future inextendible (4-dimensional) flat FLRW spacetime satisfying that $a'(0)\in (0,\infty]$. Let $\psi_{s}: (t,r) \mapsto (T,R)$ be the unique (subject to an initial condition and away from the subset $\{r^{2}a'(t)^{2}=1\}$) natural strongly spherical change of coordinates, so that $g_{s}=(\psi_{s})_{\ast}g$ can be written as in \eqref{eq:Spher1}. Then, there is no natural strongly axisymmetric $C^{0}$-extension of $(M,g)$ compatible with $\psi_s$.
\end{corollary}

\begin{proof}

Assume that $\iota : (M,g)\rightarrow (M_{\textrm{ext}},g_{\textrm{ext}})$ be a natural strongly axisymmetric $C^{0}$-extension compatible with the natural strongly spherical change of coordinates $\psi_{s}$ in $M$. 

Fix a point $p\in\partial^{-}\iota(M)$. Unwinding the definitions there
\begin{enumerate}[label=(\roman*)]
    \item  exists a neighborhood $V$ of $p$ in $M_{\textrm{ext}}$ and (potentially) a two dimensional submanifold $S\subset V$ such that $U:=V\setminus S$ satisfies $\iota^{-1}(U)\subset \tilde{M}$, where $\tilde{M}$ is as in \eqref{eq:Mwithoutbadset},
    and such that there exists a diffeomorphism $\psi_{\textrm{ext,a}}$ defined on a neighborhood $U$ such that $g_{\textrm{ext},a}=(\psi_{\textrm{ext,a}})_{\ast}g_{\textrm{ext}}$ takes the strongly axisymmetric form
    \begin{multline}
    \label{eq:SAS10}
        g_{\textrm{ext},a}=-A_{\textrm{ext}}dT_{\textrm{ext}}^{2}+B_{\textrm{ext}}dz_{\textrm{ext}}^{2}+C_{\textrm{ext}}d\rho_{\textrm{ext}}^{2}+\rho_{\textrm{ext}}^{2}d\varphi_{\textrm{ext}}^{2}+\\ 
        + \Big( (D_{1})_{\textrm{ext}}dT_{\textrm{ext}}+(D_{2})_{\textrm{ext}}dz_{\textrm{ext}}+(D_{3})_{\textrm{ext}}d\rho_{\textrm{ext}} \Big)d\varphi_{\textrm{ext}}
    \end{multline}
   on $\psi_{\textrm{ext,a}}(U)\subseteq \R^2\times (0,\infty)\times \mathbb{S}^1$ and $g_{\textrm{ext},a}$ is non-degenerate. 
    \item In $\tilde{U}=\iota^{-1}(U)\setminus \big( \{G=-\tan^2\theta \}\cup\{\theta=0,\pi\} \big) \subseteq \tilde{M}$ it holds that $\psi_{a}\circ\psi_{s}|_{\tilde{U}}=\psi_{\textrm{ext,a}}\circ\iota|_{\tilde{U}}$, where $\psi_{s}$ is the natural strongly spherical change of coordinates. 
    Thus in particular $g_{\textrm{ext},a}=g_{a}$ on $\psi_s(\tilde{U})$, where $g_{a}$ comes from the coordinate change described in equation \eqref{eq:SAS15}.
    This implies that $(A_{\textrm{ext}}, B_{\textrm{ext}}, C_{\textrm{ext}})= (A,B,C)$, $(D_{1})_{\textrm{ext}}=(D_{2})_{\textrm{ext}}=(D_{3})_{\textrm{ext}}=0$ in $\tilde{U}$,  and also $(T_{\textrm{ext}}\circ \iota ,z_{\textrm{ext}}\circ \iota,\rho_{\textrm{ext}}\circ \iota)$ have to agree with $(T, z, \rho)$ in $\tilde{U}$.
\end{enumerate}
Without loss of generality we may assume that $p\in U$: If $p$ were not itself in $U$ we could replace it by a point $q\in U\cap \partial^{-}\iota(M)$. Note that such a  $q$ must exist for dimensional reasons since $\partial^{-}\iota(M)$ is a topological hypersurface, cf.\ Theorem \ref{FLRWfutureboundary}, and $U$ is an open neighbourhood of $p\in \partial^{-}\iota(M)$ without (at most) a 
two dimensional submanifold.

Let $\gamma : (0,1]\rightarrow M$ be a past inextendible timelike curve with $\lim_{t\to0^{+}}(\iota\circ\gamma)(t)=p\in\partial^{-}\iota(M)$. In the first place, note that $\lim_{t\to0^{+}}R(\gamma(t))\neq0$ as otherwise $\rho_{\textrm{ext}}(p)=\lim_{t\to0^{+}}\rho(\gamma(t))=\lim_{t\to0^{+}}\left(R\sin{\theta}\right)(\gamma(t))=0$ contradicting that, by definition of axisymmetric coordinates, $\rho_{\textrm{ext}}\in(0,\infty)$ on $U$. By the fourth point of Remark \ref{rem:details_spherical_proof} (respectively the second part of Theorem \ref{theo:unique_spherical}) this implies that $\lim_{t\to0^{+}}G(\gamma(t))=0$ and $\gamma((0,1])\subseteq \tilde{M}$.
\par Therefore, the conditions of Theorem \ref{nocylindrical} are satisfied and we get a contradiction.
\end{proof}
Corollary \ref{cor:noaxisymmetric_Flat_FLRW} only relies on the existence of unique changes of coordinates ($\psi_{s}$ from flat FLRW to a strongly spherically symmetric spacetime for which the metric coefficient $G$ degenerates as one approaches $\pr^{-}\iota(M)$ at finite radii $R$.
As discussed in Remark \ref{rem:nospherical_hyperbolic}, also for certain hyperbolic FLRW spacetimes there exists a natural strongly spherically symmetric change of coordinates $\Tilde{\psi}_{s}$ for which the metric coefficient $G$ degenerates at $\partial^{-}\iota(M)$ as long as $R$ takes a finite non-zero limit. 
Therefore, we can also state
\begin{corollary}\label{cor:noaxisymmetric_Hyp_FLRW}
    Let $(M,g)$ be a future inextendible (4-dimensional) hyperbolic FLRW spacetime satisfying that $a'(0)\in [0,\infty]$ and $a'(0)\neq 1$. Let $\psi_{s}: (t,r) \mapsto (T,R)$ be the unique (subject to an initial condition and to a subset of the spacetime where the determinant of the Jacobian vanishes) natural strongly spherical change of coordinates, so that $g_{s}=(\psi_{s})_{\ast}g$ can be written as in \eqref{eq:Spher1}. Then, there is no natural strongly axisymmetric $C^{0}$-extension of $(M,g)$ compatible with $\psi_{s}$.
\end{corollary}
\begin{proof}
    The proof is analogous to the one of Corollary \ref{cor:noaxisymmetric_Flat_FLRW} using that in \cite{GallowayLing} (recall the discussion in Remark \ref{rem:nospherical_hyperbolic}) it was also proven that for hyperbolic FLRW spacetimes with $a'(0)\in [0,\infty]$ and $a'(0)\neq 1$ the metric coefficient $G$ vanishes as $t\to0^{+}$ provided that $R$ has a finite positive limit as $t\to0^{+}$. And this condition holds again by the same argument than the one given in the proof of the previous corollary.
\end{proof}

To end let us briefly discuss the relevance of our results in the context of the other $C^0$-inextendibility results in the literature which we reviewed in Section \ref{sec:inext_overview}. Both our results are extensions of some of the earliest $C^0$-inextendibility results through a Big Bang from \cite{GallowayLing} back in 2016, but just as their original counterparts are limited by our symmetry assumptions. One further immediately notices that the previous Corollary \ref{cor:noaxisymmetric_Hyp_FLRW} is mostly interesting for hyperbolic FLRW spacetimes without particle horizons, as if they do have a particle horizon and satisfy that $a(t)e^{\int_{t}^{1}\frac{1}{a(t')}dt'}\to \infty$ as $t\to0^{+}$, then the recent stronger general $C^{0}$-inextendibility result by Sbierski  \cite{Sbierski2023} already applies (cf.\ Theorem \ref{th:Sbierskis_inext} and Table \ref{table}). Contrary to this, Corollary \ref{cor:noaxisymmetric_Flat_FLRW} remains more widely relevant as we are still lacking any general $C^0$-inextendibility results for flat FLRW spacetimes with a Big Bang as $t\to 0^+$. Of course the symmetry assumptions remain nevertheless limiting, so trying to establish a $C^0$-inextendibility results in this case without requiring any symmetry (or at least only requiring weaker symmetry and not this strong compatibility with the symmetry structure of the original symmetry structure of the flat FLRW spacetime) in the extension continues be an interesting goal that was already formulated in \cite{GallowayLing}.\\

\end{document}